\documentclass[a4paper,12pt]{article}
\usepackage[utf8]{inputenc} 
\usepackage{amssymb}
\usepackage{graphicx}
\usepackage[english]{babel}
\usepackage{ifthen}
\usepackage{tikz} 
\usepackage{hyperref} 

\makeatletter
\def\@seccntformat#1{\csname the#1\endcsname.\ } 
\makeatother

\makeatletter

\@addtoreset{section}{language}
\makeatother

\date{}

\newenvironment{proof}[1][\hspace{-1.0ex}]%
{\par\addvspace{1mm}{\sc Proof\hspace{1.0ex}{#1}.} }%
{\quad$\blacktriangle$\par\addvspace{1mm}}

\newif\ifNoRemark
\def\addtheorem#1#2#3#4{
\ifthenelse{\equal{#2}{}}{}%
{\ifthenelse{\expandafter\isundefined\csname the#2\endcsname}{\newcounter{#2}}{}}
\newenvironment{#1}[1][\global\NoRemarktrue]
{\par\addvspace{2mm plus 0.5mm minus 0.2mm}\noindent 
{\bf #3}\ifthenelse{\equal{#2}{}}{}%
{\refstepcounter{#2}{\bf ~\csname the#2\endcsname}}%
{\bf \vphantom{##1}\ifNoRemark.\ \else\ (##1).\fi}\begingroup #4}%
{\endgroup\par\addvspace{1mm plus 0.5mm minus 0.2mm}\global\NoRemarkfalse}
\expandafter\newcommand\csname b#1\endcsname{\begin{#1}}
\expandafter\newcommand\csname e#1\endcsname{\end{#1}}
}

\addtheorem{theorem}{thrm}{Theorem}{\sl}
\addtheorem{lemma}{lmm}{Lemma}{\sl}
\addtheorem{proposition}{prpstn}{Proposition}{\sl}
\addtheorem{corollary}{crllry}{Corollary}{\sl}
\addtheorem{problem}{prblm}{Problem}{\sl}
\addtheorem{remark}{rmrk}{Remark}{}

\makeatletter
\@addtoreset{thrm}{language}
\@addtoreset{lmm}{language}
\@addtoreset{prpstn}{language}
\@addtoreset{crllry}{language}
\@addtoreset{prblm}{language}
\@addtoreset{rmrk}{language}
\makeatother

\title{Perfect codes in Doob graphs%
\thanks{The work was supported by the Russian Science Foundation (grant 14-11-00555)}%
}
\author{Denis~S.~Krotov%
\thanks{Sobolev Institute of Mathematics, Novosibirsk, Russia; Novosibirsk State University, Novosibirsk, Russia. 
E-mail: \href{mailto:krotov@math.nsc.ru}{krotov@math.nsc.ru}}%
}

\def\SSS{{\mathcal E}}
\def\TR{{\mathrm{T}}}
\def\Sh{{\mathrm{Sh}}}
\def\KK{K}
\def\DD#1#2{D(#1,#2)}
\def\oomega{{\rotatebox[origin=c]{180}{$\omega$}}}

\begin{document}
\maketitle
\begin{abstract}
We study $1$-perfect codes in Doob graphs $\DD{m}{n}$.
We show that such codes that are linear over $\mathrm{GR(4^2)}$
exist if and only if $n=(4^{\gamma+\delta}-1)/3$
and $m=(4^{\gamma+2\delta}-4^{\gamma+\delta})/6$ for some integers
$\gamma \ge 0$ and $\delta>0$.
We also prove necessary conditions on $(m,n)$
for $1$-perfect codes that are linear over $\mathbb{Z}_4$
(we call such codes additive) to exist in $\DD{m}{n}$ graphs;
for some of these parameters,
we show the existence of codes.
For every $m$ and $n$ satisfying 
$2m+n = (4^\mu - 1)/3$ and 
$m\leq (4^\mu-5\cdot 2^{\mu-1}+1)/9$, 
we prove the existence of $1$-perfect codes in $\DD{m}{n}$,
without the restriction to admit some group structure.


\end{abstract}

\section{Introduction}\label{s:intro}

A connected regular graph is called \emph{distance regular}
if every bipartite subgraph 
generated by two cocentered spheres of different radius
is biregular. 
A set of vertices of a graph or any other discrete metric space is called an \emph{$e$-perfect code}, 
or simply a \emph{perfect code}, if the vertex set
is partitioned into the radius-$e$ balls centered in the code vertices.
The codes of cardinality $1$ and the $0$-perfect codes are called \emph{trivial} perfect codes.

The perfect codes in distance regular graphs are objects 
that are highly interesting from the point 
of view of both coding theory and algebraic combinatorics.
On one hand, these codes are error correcting codes that attain
the sphere-packing bound (``perfect'' means ``extremely good'').
On the other hand, they possess algebraic properties that are connected
with the algebraic properties of the distance regular graph;
a perfect code is a some kind of divisor \cite[Ch.~4]{CDS} of the graph.

It may safely be said that the most important class of distance regular graphs, 
for coding theory, is the Hamming graphs. 
The \emph{Hamming graph} $H(n,q)$ 
is the Cartesian product 
of $n$ copies of the complete graph of order $q$.
For the Hamming graphs $H(n,q)$, 
the study of possible parameters of perfect codes 
is completed only if $q$ is a prime power.
In this case, as was shown in \cite{Tiet:1973,ZL:1973}, 
there are no nontrivial perfect codes except 
the $1$-perfect codes in $H((q^m-1)/(q-1),q)$ \cite{Hamming:50,Golay:49}, 
the $3$- and $2$-perfect Golay codes in $H(23,2)$ and $H(11,3)$, 
respectively \cite{Golay:49},
and the $e$-perfect binary repetition codes in $H(2e+1,2)$.
In the case of non-prime-power $q$, no nontrivial perfect codes are known,
and the parameters for which the nonexistence is not proven are restricted
by $1$-perfect codes and $2$-perfect codes 
(the last case is solved for some values of $q$),
see \cite{Heden:2010:non-prime} for a survey of the known results in this area.

We briefly mention two other infinite classes of distance regular graphs of unbounded diameter
that occur in coding theory applications. 
The Johnson graph $J(n,w)$ can be considered
as the distance-$2$ graph of a radius-$w$ sphere in $H(n,2)$.
The well-known Delsarte conjecture states that there are
no nontrivial perfect codes in the Johnson graphs.
In general, the problem is open; we refer \cite{Etzion:2007:Johnson} 
for a survey of known nonexistence results and mention a later result \cite{Gordon:2006}, 
where the nonexistence of $1$-perfect codes in $J(n,w)$ is computationally proved
for ``small'' values of $n \le 2^{250}$.
The nonexistence of nontrivial perfect codes in the Grassmann graphs $J_q(n,w)$
was proven in \cite{Chihara87}; 
a relatively simple proof can be found in \cite{MarZhu:1995}.

The Doob graph $\DD{m}{n}$ is a distance regular graph of diameter $2m+n$ 
with the same parameters 
as the Hamming graph $H(2m+n,4)$. 
As noted in \cite{KoolMun:2000},
nontrivial $e$-perfect codes in $\DD{m}{n}$ can only exist
when $e=1$ and $2m+n=(4^\mu-1)/3$ for some integer $\mu$
(with exactly the same proof as for $H(2m+n,4)$).
In \cite{KoolMun:2000}, Koolen and Munemasa constructed 
$1$-perfect codes in the Doob graphs of diameter $5$.

In the current paper, we show the existence of $1$-perfect codes
in $\DD{m}{n}$ in approximately two-thirds (as $\mu\to\infty$) 
of possible values of $(m,n)$
satisfying $2m+n=(4^\mu-1)/3$.
Additionally, we study the existence of linear, 
over the rings $\mathrm{GR(4^2)}$ and $\mathbb{Z}_4$,
$1$-perfect codes in Doob graphs.

In Section~\ref{s:shrik}, we define the Doob graphs
with underlying structure of a module over the ring
$\mathrm{GR}(4^2)$ or $\mathbb{Z}_4$;
also, we define linear (over $\mathrm{GR}(4^2)$)
and additive (over $\mathbb{Z}_4$) codes.
In Section~\ref{s:par}, we prove some restrictions
on the parameters of a Doob graph that can contain an additive $1$-perfect code,
in terms of parameters $\Gamma$, $\Delta$ of the factorgroup
$\mathbb{Z}_2^\Gamma \times\mathbb{Z}_4^\Delta$
of cosets of the code. 
The proof exploits ideas from \cite{BorRif:1999}.
In Section~\ref{s:constr}, 
we construct linear $1$-perfect codes 
for each admissible parameters.
In Section~\ref{s:add},
we construct additive $1$-perfect codes for each set of parameters
meeting the conditions of Section~\ref{s:par} with even $\Delta$.
In Section~\ref{s:3}, we construct an example 
of additive $1$-perfect code
with odd $\Delta =3$.
In Section~\ref{s:nonl}, we construct $1$-perfect codes
in $\DD{m}{n}$ for each admissible diameter $2m+n$ 
and small $m$ (approximately, $m \lesssim n$).
In the last section, we list open problems concerning the existence on $1$-perfect codes in Doob graph.
\section{Representation of the Doob graphs}\label{s:shrik}
Let $\mathbb{Z}$ denote the ring of integers, and let $\mathbb{Z}_p = \mathbb{Z} / p\mathbb{Z}$
denote the factor-ring of residue classes of $\mathbb{Z}$ modulo $p$.
If $\mathbb{M}$ is a ring or a module over a ring, then $\mathbb{M}^+$ denotes the additive group of $\mathbb{M}$.
The \emph{Eisenstein integers} $\mathbb{E}$ are the complex numbers of the form
$$ z = a + b \omega, \qquad 
\omega = \frac{-1+i\sqrt{3}}2 = e^{2\pi i/3}, \quad a,b\in \mathbb{Z}.$$
Given $p\in \mathbb{E} \backslash \{0\}$, we denote 
by $\mathbb{E}_p$ the ring $\mathbb{E}/p\mathbb{E}$ 
of residue classes of $\mathbb{E}$ modulo $p$. 
We are interested in the two cases 
$\mathbb{E}_2$ and $\mathbb{E}_4$
(see Fig.~\ref{fig:1}).

$\mathbb{E}_2$ is the Galois field 
$\mathrm{GF}(2^2)$ of characteristic $2$. 
Its elements are $[0]_2$, $[1]_2$, $[\omega]_2$, and $[\oomega]_2$, where $\oomega=\omega^2$,
and $[x]_p=x+p\mathbb{E}$;
but in what follows, we will omit the braces $[\ ]_p$
when naming the residue classes from $\mathbb{E}_p$, $p=2,4$.

$\mathbb{E}_4$ is the Galois ring 
$\mathrm{GF}(4^2)$ of characteristic $4$.
Its elements are 
$2b+a$, $a,b\in \{0, 1, \omega, \oomega\}$.
The set of units 
$\{ 1, -\omega, 
\oomega, -1, 
\omega, -\oomega \}$ 
will be denoted by $\SSS$.
\begin{lemma}\label{l:4cosets}
The set of all elements of $\mathbb{E}_4$ is partitioned into four 
multiplicative cosets of $\SSS$:
\begin{eqnarray*}
0\SSS &=& \{0\}, \\
\SSS &=& \{1,\ 2\omega+\omega,\ 
\oomega,\ 2+1,\ 
\omega,\ 2\oomega+\oomega\} 
\quad\mbox{(Fig.~\ref{fig:1}, solid circle)},\\
2\SSS &=& \{2,\ 2\omega,\ 2\oomega\}, \\
\psi \SSS &=& \{ 2+\omega,\ 2\omega+1,\ 2\omega+\oomega,\ 
2\oomega+\omega,\ 2\oomega+1,\ 2+\oomega\}
\quad\mbox{(Fig.~\ref{fig:1}, dashed circle)},
\end{eqnarray*}
where $\psi$ is an arbitrary representative of the corresponding coset, 
say, $\psi = 2+\omega$. 
The set $2\SSS$ is exactly the set of nontrivial zero divisors 
of the ring $\mathbb{E}_4$, 
while the set of regular elements is $\SSS \cup \psi\SSS$.
\end{lemma}

\begin{figure}[htb]
\noindent\mbox{}\hfill
\begin{tikzpicture}[
cell/.style={
fill=white,
fill opacity=0.7,
thin,
anchor=south,
},
cll/.style={
thin,
anchor=south,
},
nz/.style={circle,fill=white,draw=black, 
           inner sep=1.5pt},
zz/.style={circle,fill=black,draw=black, 
           inner sep=1.5pt},
scale=1.2]
\begin{scope}
\clip (-2.7,-1.95) rectangle (4.2,2.2);
\draw[xslant=0.577,ystep=.866,xstep=1,dotted] (-4.9,-2.1) grid (5.4,3.1);
\draw[xslant=-0.577,ystep=9.866,xstep=1,dotted] (-3.4,-2.1) grid (6.4,3.1);
\end{scope}
\begin{scope}
\clip [xslant=-0.577] (-1.4,-1.25) rectangle (2.4,2.2);
\draw[xslant=0.577,ystep=.866,xstep=1,draw=yellow!80!green,very thick] (-4.9,-2.1) grid (5.4,3.9);
\draw[xslant=-0.577,ystep=9.866,xstep=1,draw=yellow!80!green,very thick] (-3.4,-2.1) grid (6.4,3.9);
\end{scope}
\draw[draw=black,->] (-2.7,0) -- (4.8,0) node[anchor=north east]{$\scriptstyle\mathrm{Re}$};
\draw[draw=black,->] (0,-1.95) -- (0,2.6) node[anchor=north east]{$\scriptstyle\mathrm{Im}$};
\draw [thick,draw=blue] (0,0) circle (1);
\draw [thick,draw=blue,densely dashed] (0,0) circle (1.72);
\foreach \angl in {0,120,240}
{\draw [draw=blue] (\angl:2) +(120:0.15) -- +(180:0.15) -- +(240:0.15) -- +(300:0.15) -- +(0:0.15) -- +(60:0.15);}
\draw (-4,0) \foreach \mrk in {$2$,$0$,$2$,$0$} {++(2,0)  node [cell] {\mrk} node [cll] {\mrk} node [zz] {}}
(-3,0) \foreach \mrk in {$3$,$1$,$3$} {++(2,0)  node [cell] {\mrk} node [cll] {\mrk} node [nz] {}};
\draw (120:2) ++(-3,0) \foreach \mrk in {$2\oomega{+}1$,$2\omega{+}1$,$2\oomega{+}1$,$2\omega{+}1$} {++(2,0)  node [cell] {\mrk} node [cll] {\mrk} node [nz] {}}
++(-7,0) \foreach \mrk in {$2\omega$,$2\oomega$,$2\omega$} {++(2,0)  node [cell] {\mrk} node [cll] {\mrk} node [zz] {}};
\draw (240:2) ++(-3,0) \foreach \mrk in {$2\omega{+}1$,$2\oomega{+}1$,$2\omega{+}1$,$2\oomega{+}1$} {++(2,0)  node [cell] {\mrk} node [cll] {\mrk} node [nz] {}}
++(-7,0) \foreach \mrk in {$2\oomega$,$2\omega$,$2\oomega$} {++(2,0)  node [cell] {\mrk} node [cll] {\mrk} node [zz] {}};
\draw 
(120:1) ++(-3,0) \foreach \mrk in {$2{+}\omega$,$2\omega{+}\oomega$,$\omega$,$3\oomega$,$2{+}\omega$,$2\omega{+}\oomega$,$\omega$} {++(1,0)  node [cell] {\mrk} node [cll] {\mrk} node [nz] {}};
\draw 
(240:1) ++(-3,0) \foreach \mrk in {$2{+}\oomega$,$2\oomega{+}\omega$,$\oomega$,$3\omega$,$2{+}\oomega$,$2\oomega{+}\omega$,$\oomega$} {++(1,0)  node [cell] {\mrk} node [cll] {\mrk} node [nz] {}};
\draw (0,0) node [circle,fill=blue,draw=black, 
           inner sep=2pt] {};
\fill [xslant=-0.577,fill=black,opacity=0.03] (-1.4,-1.25) rectangle (2.4,2.2);
\end{tikzpicture}
\hfill\mbox{}
\caption{Representation of the Shrikhande graph
as a Cayley graph of $\mathbb{E}/4\mathbb{E}\simeq\mathrm{GR}(4^2)$. 
The solid circle indicates the group of units $\SSS$;
the dashed circle indicates the coset $\psi\SSS$;
the three small hexagons indicate the coset $2\SSS$.}\label{fig:1}
\end{figure}
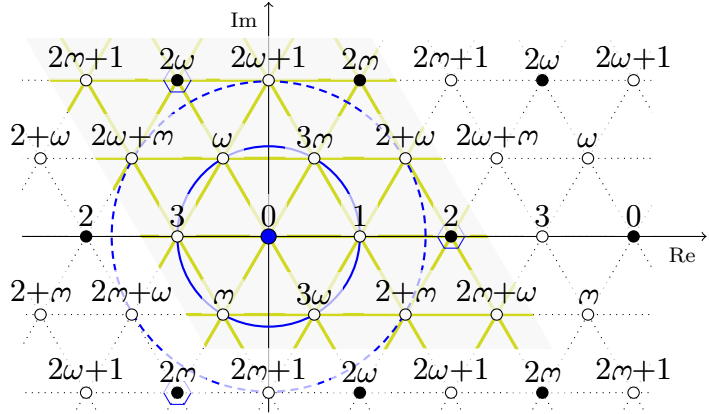

The \emph{Shrikhande graph} $\Sh$
is the Cayley graph of the additive group $\mathbb{E}_4^+$ of $\mathbb{E}_4$
with the generating set $\SSS$. That is, the vertex 
set is the set of elements of $\mathbb{E}_4$, two elements being adjacent
if and only if their difference is in $\SSS$.

The ring $\mathbb{E}_4$ itself can be considered as a module of type
$\mathbb{Z}_4^2$ over $\mathbb{Z}_4$. 
Every element $x$ of
$\mathbb{E}_4$ can be represented by a pair of coordinates 
in the basis  $(\omega, 1)$; denote this pair by $\widehat x$.
By $\widetilde x$, we denote the $2\times 2$ matrix over $\mathbb{Z}_4$ that correspond to the multiplication 
by $x$ in $\mathbb{E}_4$; that is, $z=xy$ is equivalent to $\widehat z^\TR = \widetilde x \widehat y^\TR$.
The Cayley graph of $\mathbb{Z}_4^{2+}$ with the generating set
$\widehat\SSS = \{\widehat 1, -\widehat\omega, \widehat\oomega, -\widehat 1, \widehat\omega, -\widehat\oomega\}= \{01,30,33,03,10,11\}$
will be denoted by $\Sh$, too.

We will use three different representations of 
the full $4$-vertex graph $\KK = K_4$ as a Cayley graph.
At first, it will be considered as 
the Cayley graph of $\mathbb{E}_2^+$
with the generating set $\{1,\omega, \oomega\}$.
Similar to the case of $\mathbb{E}_4$, 
we can treat $\mathbb{E}_2$ as a $2$-dimensional vector space
over the field $\mathbb{Z}_2$ and name its elements by
the pairs of coordinates in the basis  $(\omega, 1)$ 
(we will use the notations $\widehat x$ and $\widetilde x$ in this case as well).
This gives the second representation of $\KK$ as 
the Cayley graph of $\mathbb{Z}_2^{2+}$ 
with the generating set $\{01,10,11\}$.
At third, $K$ will be considered as 
the Cayley graph of $\mathbb{Z}_4^+$ 
with the generating set $\{1,2,3\}$.

Denote by $\DD{m}{n}$ the Cartesian product  $\Sh^m \times \KK^n$ 
of $m$ copies
of the Shrikhande graph and $n$ copies 
of the full $4$-vertex graph.
If $m>0$, then $\DD{m}{n}$ is called a \emph{Doob graph};
the case $m=0$ corresponds to the Hamming graph $H(n,4)$.
Accordingly with different representations of $\Sh$ and $K$,
we will consider two representations
of the vertex set of $\DD{m}{n}$. 

At first, it is the set of $(m+n)$-tuples
$(x_1,{\ldots},x_m,y_1,{\ldots},y_n)$ from 
$\mathbb{E}_4^m\times\mathbb{E}_2^n$, 
which is a module over the ring $\mathbb{E}_4$
(the addition and the multiplication by a constant from $\mathbb{E}$ is defined coordinatewise, 
modulo $4$ in the first $m$ coordinates and modulo $2$ in the last $n$ coordinates).
We call a code $C \subset \mathbb{E}_4^m\times\mathbb{E}_2^n$
\emph{linear} if it is a submodule, that is, it is closed 
with respect to addition and multiplication by an element of $\mathbb{E}_4$.

At second, we can take the set of $(2m+2n'+n'')$-tuples
$(x_1,{\ldots},x_{2m},y_1,{\ldots},y_{2n'},z_1,{\ldots},z_{n''})$ from
$\mathbb{Z}_4^{2m}\times\mathbb{Z}_2^{2n'}\times\mathbb{Z}_4^{n''}$, 
$n'+n''=n$, as the vertex set of $\DD{m}{n}$.
If a code 
$C \subset \mathbb{Z}_4^{2m}\times\mathbb{Z}_2^{2n'}\times\mathbb{Z}_4^{n''}$
is closed with respect to addition, then we call it \emph{additive}.
An additive code is necessarily closed 
with respect to multiplication by an element of $\mathbb{Z}_4$;
so, it is in fact a submodule of the module
$\mathbb{Z}_4^{2m}\times\mathbb{Z}_2^{2n'}\times\mathbb{Z}_4^{n''}$
over $\mathbb{Z}_4$.

The natural graph distance in $\DD{m}{n}$
provides a metric on $\mathbb{E}_4^m\times\mathbb{E}_2^n$
or 
$\mathbb{Z}_4^{2m}\times\mathbb{Z}_2^{2n'}\times\mathbb{Z}_4^{n''}$,
which will be called the \emph{$\DD{m}{n}$-metric} 
(if $m>0$, a \emph{Doob metric}; if $m=0$, the \emph{Hamming metric}).
The \emph{weight} of a vertex $x$ of $\DD{m}{n}$ 
is the distance from $x$ to $\overline 0$ 
(here and in what follows, $\overline 0$ denotes the zero element of the module,
i.e., the all-zero tuple, whose length is clear from the context).

If we study $1$-perfect codes, the vertices of weight $1$ are of special interest.
Recall that in the case of $\mathbb{E}_4^m\times\mathbb{E}_2^n$,
these vertices are the tuples with only one nonzero element, which belongs to $\SSS$
if it is placed in the $\mathbb{E}_4$-part of the tuple and belongs to $\{1,\omega,\oomega\}$
if its position is in the $\mathbb{E}_2$-part.
In the case of $\mathbb{Z}_4^{2m}\times\mathbb{Z}_2^{2n'}\times\mathbb{Z}_4^{n''}$ with $\DD{m}{n'+n''}$-metric,
every vertex of weight $1$ has one of the forms
$(0{\ldots}0xy0{\ldots}0|\overline 0|\overline 0)$,
$(\overline 0|0{\ldots}0vw0{\ldots}0|\overline 0)$,
$(\overline 0|\overline 0|0{\ldots}0z0{\ldots}0)$, where $x$ and $v$ are in odd positions,
$xy\in\{01,11,10,03,33,30\}$, $vw\in\{01,11,10\}$, $z\in\{1,2,3\}$,
and the vertical lines separate the three parts of the tuple of length $2m$, $2n'$, and $n''$, respectively.

\section{Restrictions on the parameters of additive codes}\label{s:par}
In this section, we derive restrictions 
on the parameters $m$, $n'$, $n''$
of the Doob graph $\DD{m}{n'+n''}$ 
containing an additive $1$-perfect code.
Construction of codes for a wide class (but not for all) 
of parameters satisfying the derived restrictions will be suggested in the next three sections.

\begin{theorem}\label{th:param}
Assume that there is an additive $1$-perfect code in 
$\mathbb{Z}_4^{2m}\times\mathbb{Z}_2^{2n'}\times\mathbb{Z}_4^{n''}$ with 
the Doob $\DD{m}{n'+n''}$-metric.
Then $n''\ne 1$ and for some even 
$\Gamma\ge 0$ and integer $\Delta \ge 2$,
\begin{eqnarray}
2m+n'+n'' &=& (2^{\Gamma+2\Delta}-1)/3, \label{eq:1} \\
3n'+n''&=&2^{\Gamma+\Delta}-1, \label{eq:2} \\
n'' & \le &  2^{\Delta}-1 \label{eq:3}
\end{eqnarray}
\end{theorem}
\begin{proof}
Assume $C\subset \mathbb{Z}_4^{2m}\times\mathbb{Z}_2^{2n'}\times\mathbb{Z}_4^{n''}$ 
is an additive $1$-perfect code in $\DD{m}{n}$.
For every weight-$1$ vertex $e$,
the set $[e]=e+C$ is also a $1$-perfect code 
(this follows from the general fact that addition a constant preserves the distance,
which is true for any Cayley graph).
As follows from the definition of $1$-perfect code, 
the set of all such $[e]$, together with $C$ itself,
form a partition of the module; hence, they form 
the factorgroup $(\mathbb{Z}_4^{2m}\times\mathbb{Z}_2^{2n'}\times\mathbb{Z}_4^{n''})^+ / C$.
This group is isomorphic to $(\mathbb{Z}_2^\Gamma \times \mathbb{Z}_4^\Delta)^+$
for some nonnegative integers $\Gamma$ and $\Delta$.
The number of elements of order $2$ in this group is $2^{\Gamma+\Delta}-1$.
On the other hand, the number of order-$2$ tuples of weight $1$
is $3n'+n''$. 
Moreover, if $e$ is an order-$4$ tuple of weight $1$,
then $e+e$ does not coincide with $\overline 0$,
is adjacent to $e$, and thus cannot belong to $C$, 
which means that $[e]$ has order $4$ in the factorgroup as well.
So,  $3n'+n''$ is also the number of elements of order $2$ in the factorgroup, and (\ref{eq:2}) holds. 
Additionally,
as the order of the factorgroup
coincides with the number of weight-$1$ vertices plus one,
we get $2^{\Gamma+2\Delta}=6m+3(n'+n'')+1$, i.e. (\ref{eq:1});
we also note that this equation has integer solutions only for even $\Gamma$.
To prove the inequality (\ref{eq:3}), we note that $n''$ weight-$1$ vertices
have the form $2e$ for some $e$; hence, the same is true for the corresponding 
cosets. But the number of such nonzero elements in the factorgroup is $2^{\Delta}-1$;
so, $n''$ cannot exceed this value.

It remains to prove that $n'' \ne 1$. Assume the contrary, $n'' = 1$.
Consider the set of all $2^{\Gamma+\Delta}-1$ order-$2$ elements of the factorgroup.
It is partitioned into $n'$ triples of elements $[e_{2m+2i-1}]$, $[e_{2m+2i}]$, $[e_{2m+2i-1}+e_{2m+2i}]$, 
$i=1,{\ldots},n'$,
and one additional element $[2e_{2m+2n'+1}]$,
where $e_{j}$ is the tuple 
with one in the $j$th position 
and zeros in the others.
We see that the sum of all order-$2$ elements is $[2e_{2m+2n'+1}]$,
i.e., non-zero, which is obviously impossible if 
$\Gamma+\Delta>1$. The case $\Gamma+\Delta=1$ is degenerate and yields 
$m=0$, which is not allowed by the definition of a Doob graph.

Finally, we note that 
$\Delta=0$
implies $m=0$, which is not allowed by the definition of a Doob graph,
and $\Delta=0$ implies $n''=1$ which is proven to be impossible.
\end{proof}
\begin{corollary}\label{c:param}
Assume that there is a linear $1$-perfect code in 
$\mathbb{E}_4^{m}\times\mathbb{E}_2^{n}$ or an additive $1$-perfect code in 
$\mathbb{Z}_4^{2m}\times\mathbb{Z}_2^{2n}$ with 
the $\DD{m}{n}$-metric.
Then
 for some integers 
$\gamma\ge 0$ and $\delta>0$,
$$n=(4^{\gamma+\delta}-1)/3 \qquad \mbox{and} \qquad
m=(4^{\gamma+2\delta}-4^{\gamma+\delta})/6.$$
\end{corollary}
\begin{proof}
In the case $n'=n$, $n''=0$, 
the solution of the equations from the statement of Theorem~\ref{th:param} is
$n=(2^{\Gamma+\Delta}-1)/3$, 
$m=(2^{\Gamma+2\Delta}-2^{\Gamma+\Delta})/6$. Since $m$ and $n$ are integers
only if both $\Gamma$ and $\Delta$ are even, we get the statement with $\gamma=\Gamma/2$ and
$\delta=\Delta/2$.
\end{proof}

\begin{remark}\label{rem:m0}
Although we formally require that $m>0$ for Doob graphs, the arguments in this section
still work for the case $m=0$. As a result, from (\ref{eq:1})--(\ref{eq:3}) we can see
that nontrivial additive $1$-perfect codes in $\mathbb{Z}_2^{2n'}\times \mathbb{Z}_4^{n''}$ with 
the Hamming $\DD{0}{n'+n''}$-metric can only exist when $n''=0$.
The results in the next section (see also Corollary~\ref{cor:add0-constr}) 
can also be applied to the degenerated case $m=\delta=0$,
providing a construction of such codes, 
which are well known Hamming $4$-ary codes.
\end{remark}
\section{Construction of linear codes}\label{s:constr}
Let $\gamma \ge 0$ and $\delta>0$ be integers. 
Consider two matrices 
$A^*=A^*_{\gamma,\delta}$ and $A'=A'_{\gamma,\delta}$.
The matrix $A^*$ consists of all columns from $\mathbb{E}_4^{\gamma+\delta}$ satisfying the following:

(*) the order of the column is $4$;

(**) the first regular (order-$4$) element of the column is either $1$ or $\psi=2+\omega$;

(***) the last $\gamma$ elements of the column are zero divisors.

The number of such columns is $(16^{\delta}4^{\gamma}-4^{\delta+\gamma})/6$, which will be denoted by $m$;
so, $A^*$ is a $({\gamma+\delta}) \times m$ 
matrix over $\mathbb{E}_4$.

The matrix $A'$ 
consists of all $n=(4^{\delta+\gamma}-1)/3$ 
nonzero columns from $\mathbb{E}_2^{\gamma+\delta}$ whose first nonzero element is $1$.

We now merge the matrices $A^*$ and $A'$ into the matrix $A=A_{\gamma,\delta}=A^*|A'$ of size $({\gamma+\delta}) \times (m+n)$
and define the multiplication $A z^{\TR}$ for $z = (x|y)\in \mathbb{E}_4^m\times \mathbb{E}_2^n$ as
$A^* x^{\TR}+2A' y^{\TR}$ (here, the result of the multiplication by $2$ is considered as a column-vector over $\mathbb{E}_4$). For example, 
\begin{eqnarray*}
A_{0,2}=\left(\begin{array}{ccccccccccccccccccccccccccc|ccccc}
  0&0&2&2&2\omega&2\omega&2\oomega&2\oomega&
 1&1&1&1&1&1&1&1&1&1
\\
1&\psi&1&\psi&1&\psi&1&\psi&
0&2&2\omega&2\oomega&
1&-\oomega&\omega&-1&\oomega&-\omega
\\ \hline\end{array}\right.\ \ 
\\ 
\left.\begin{array}{ccccccccc|ccccc}
  1&1&1&1&1&1&
\psi&\ldots&\psi& 
 0&1&1&1&1
 \\
 2{+}\omega&2\omega{+}1&2\omega{+}\oomega& 
2\oomega{+}\omega& 2\oomega{+}1& 2{+}\oomega& 0 &\ldots&2{+}\oomega& 1&0&1&\omega&\oomega
\\ \hline\end{array}\right),
\end{eqnarray*}
$$
A_{1,1}=\left(\begin{array}{cccccccc|ccccc}
1&1&1&1&
 \psi&\psi&\psi&\psi& 
 0&1&1&1&1  
 \\ \hline
 0&2&2\omega&2\oomega&0&2&2\omega&2\oomega&
1&0&1&\omega&\oomega
\end{array}\right),
\qquad A_{0,1}=
\left(\begin{array}{cc|c}
1&\psi&1 \\ \hline
\end{array}\right).
$$

\begin{theorem}\label{th:lin-constr}
Let the matrix $A=A_{\gamma,\delta}$ be constructed as above.
The set $C=C_{\gamma,\delta} = \{c\in \mathbb{E}_4^m\times \mathbb{E}_2^n \,:\, Ac^{\TR} = \overline 0^{\TR}\}$
is a linear $1$-perfect code in the Doob graph $\DD{m}{n}$.
\end{theorem}
\begin{proof}
For a tuple $z\in \mathbb{E}_4^m\times \mathbb{E}_2^n$,
the value $Az^{\TR}$ is called a syndrome of $z$.
Note that the last $\gamma$ elements of every syndrome are
divisors of zero; so, there are at most $16^\delta 4^\gamma$
different syndromes.
Let us consider an arbitrary 
$z\in \mathbb{E}_4^m\times \mathbb{E}_2^n$ 
and its syndrome $s=Az^{\TR}$. 
If $s$ is the all-zero column, 
then $z\in C$. 
Let us show that if $s$ is non-zero,
then there is a unique codeword $c=z-e$ adjacent to $z$.
For the existence, it is sufficient to find a weight-$1$ tuple $e$ 
with syndrome $s$. 
We will say that $s$ is \emph{covered} by the coordinate $i$
if it is the syndrome of some $e$ of weight $1$ 
with the only non-zero value in the position $i$.
Let us consider two cases.

(i) If $s$ is of order $2$, then, 
by the definition of $A'$ and Lemma~\ref{l:4cosets}, 
$s$ is representable as $2\alpha a$ for some column $a$ of $A'$, where
$\alpha$ from $\{1,\omega, \oomega\}$ is the first non-zero element of $s$.
Then $s$ is covered by the corresponding coordinate.

(ii) If $s$ is of order $4$, then, 
by the definition of $A^*$ and Lemma~\ref{l:4cosets}, 
$s$ is representable as $\beta b$ for the column $b=s/\beta$ of $A^*$, where
$\beta\in\SSS$ and the first regular element of $s$ is $\beta$ or $\psi\beta$.
Then, again, $s$ is covered by the corresponding coordinate.

It is easy to see also that the choice of $e$ is unique 
(which also follows from numerical reasons: the number of weight-$1$ tuples coincide
with the number of possible syndromes).
Then, $C$ is a $1$-perfect code by the definition.
\end{proof}
The matrix $A$, defining the code $C$ as the kernel of the corresponding homomorphism, is known as a \emph{check matrix} of $C$.

\section{Construction of additive codes, even $\Delta$}\label{s:add}
The linear codes constructed in the previous section are trivially 
additive codes in $\mathbb{Z}_4^{2m}\times\mathbb{Z}_2^{2n}$,
if we treat the elements of $\mathbb{E}_4$ and $\mathbb{E}_2$
as vectors over $\mathbb{Z}_4$ and $\mathbb{Z}_2$, respectively.

\begin{corollary}\label{cor:add0-constr}
Let the matrix $B$ be obtained from the matrix $A$ constructed in Section~\ref{s:constr} by replacing 
every item $x$ by the $2\times 2$ matrix $\widetilde x$, over $\mathbb{Z}_4$ or $\mathbb{Z}_2$.
The set $\widehat C = \{c\in \mathbb{Z}_4^{2m}\times \mathbb{Z}_2^{2n} \,:\, Bc^{\TR} = \overline 0^{\TR}\}$, where
$B(x|y)^\TR=B^*x^{\TR}+2B'y^\TR$,
is an additive $1$-perfect code in the Doob graph $\DD{m}{n}$.
\end{corollary}

To construct additive codes in 
$\mathbb{Z}_4^{2m}\times\mathbb{Z}_2^{2n'}\times\mathbb{Z}_4^{n''}$
with $n''>0$, we will start from the check matrix $B$ of the code $\widehat C$, remove 
some columns from the first $\mathbb{Z}_4$- and second $\mathbb{Z}_2$- parts of the matrix and add
columns to the new, third, $\mathbb{Z}_4$-part of the matrix.

Let the matrix $B = B^*|B'$ be constructed from the matrix $A=A^*|A'$ as in Corollary~\ref{cor:add0-constr}.
Let $\lambda_1^\TR$, {\ldots}, $\lambda_{n''/3}^\TR$ be some columns of $A'$ having zeros in the last $\gamma$ positions (by (\ref{eq:3}), there are at least $n''/3$ such columns, while by (\ref{eq:2}) this number is integer%
).
Note that $A^*$ also has the same columns, but treated as vectors over $\mathbb{E}_4$.
Let the matrices $D^*$ and $D'$ be obtained from $B^*$ and $B'$, respectively,
by removing the corresponding $2n''/3$ columns. 
And let $D''$ be the matrix with the columns 
$\widehat{\lambda}_1^\TR$, $\widehat{\omega\lambda}_1^\TR$, $\widehat{\oomega\lambda}_1^\TR$,
{\ldots}, $\widehat{\lambda}_{n''/3}^\TR$, $\widehat{\omega\lambda}_{n''/3}^\TR$, $\widehat{\oomega\lambda}_{n''/3}^\TR$.
Denote $D=D^*|D'|D''$. The following example illustrates the transformation $A \to B \to D$ ($\gamma=0$, $\delta=2$).
$$
\left(\begin{array}{@{\ .\,.\ }c@{\ .\,.\ }|@{\ .\,.\ }c@{\ .\,.\ }}
1&1 \\
\oomega & \oomega
\end{array}\right) \quad\longrightarrow\quad
\left(\begin{array}{@{\ .\,.\ }c@{\,}c@{\ .\,.\ }|@{\ .\,.\ }c@{\,}c@{\ .\,.\ }}
1&0&1&0 \\[-0.3ex]
0&1&0&1 \\[0.3ex]
0&3&0&1 \\[-0.3ex]
1&3&1&1
\end{array}\right) \quad\longrightarrow\quad
\left(\begin{array}{@{\ .\,.\ }|@{\ .\,.\ }|@{\ .\,.\ }c@{\ }c@{\ }c@{\ .\,.\ }}
0&1&3 \\[-0.3ex]
1&0&3 \\[0.3ex]
3&0&1 \\[-0.3ex]
3&1&0
\end{array}\right)
$$
\begin{theorem}\label{th:add-constr}
Let the matrix $D=D^*|D'|D''$ be defined as above.
Then the set 
$\overline C = \{c\in \mathbb{Z}_4^{2m^*}
\times \mathbb{Z}_2^{2n'}
\times \mathbb{Z}_4^{n''}
\,:\, Dc^{\TR} = \overline 0^{\TR}\}$, 
where $D(x|y|z)^\TR = D^*x^{\TR}+2D'y^\TR+D''z^\TR $,
is an additive $1$-perfect code in $\DD{m^*}{n'+n''}$.
\end{theorem}
\begin{proof}
As in the proof of Theorem~\ref{th:lin-constr}, 
given a check matrix, we will say
that some syndrome $s$ is covered by some coordinates if there is a weight-$1$ tuple $e$
with zeros out of these coordinates and with the syndrome $s$.

We first consider the check matrix $A=A^*|A'$.
Consider a column 
$\lambda_i^\TR$ of $A'$ having zeros in the last $\gamma$ positions.
The corresponding coordinate covers three syndromes, 
$2\lambda_i^\TR$, $2\omega\lambda_i^\TR$, and $2\oomega\lambda_i^\TR$.
Hence, the corresponding two columns of the matrix $B$ cover the three syndromes
$2\widehat{\lambda_i}^\TR$, 
$2\widehat{\omega\lambda_i}^\TR$, 
$2\widehat{\oomega\lambda_i}^\TR$.
Next, consider the column $\lambda_i^\TR$ of $A^*$. 
The corresponding coordinate covers six syndromes, 
$\lambda_i^\TR$, 
$\omega\lambda_i^\TR$, 
$\oomega\lambda_i^\TR$,
$3\lambda_i^\TR$, 
$3\omega\lambda_i^\TR$, 
and $3\oomega\lambda_i^\TR$.
Hence, the corresponding two columns of the matrix $B$ cover the six syndromes
$\widehat{\lambda_i}^\TR$, 
$\widehat{\omega\lambda_i}^\TR$, 
$\widehat{\oomega\lambda_i}^\TR$,
$3\widehat{\lambda_i}^\TR$, 
$3\widehat{\omega\lambda_i}^\TR$, 
$3\widehat{\oomega\lambda_i}^\TR$.

Now consider the matrix $D=D^*|D'|D''$. The coordinate, corresponding to the column 
$\widehat{\lambda_i}^\TR$ of $D''$, covers the three syndromes
$\widehat{\lambda_i}^\TR$, $2\widehat{\lambda_i}^\TR$, and $3\widehat{\lambda_i}^\TR$.
The coordinates, corresponding to the columns
$\widehat{\omega\lambda_i}^\TR$ and $\widehat{\oomega\lambda_i}^\TR$ of $D''$, covers the syndromes
$\widehat{\omega\lambda_i}^\TR$, $2\widehat{\omega\lambda_i}^\TR$, $3\widehat{\omega\lambda_i}^\TR$ 
and
$\widehat{\oomega\lambda_i}^\TR$, $2\widehat{\oomega\lambda_i}^\TR$, $3\widehat{\oomega\lambda_i}^\TR$,
respectively.

We see that after removing the four columns from the matrix $B$ and adding the three columns to the new,
third part of the check matrix, the set of covered syndromes has not been changed. 
As it is true for every $i$ from $1$ to $n''$, with the check matrix $D$, every syndrome is covered.
Moreover, by the numerical reasons, every nonzero syndrome is the syndrome of a unique weight-$1$ vertex.
This proves that the code is $1$-perfect.
\end{proof}
\begin{corollary}\label{cor:add-exist}
For every $m$, $n'$, and $n''$ satisfying the statement of Theorem~\ref{th:param} with even $\Delta$,
there is a $1$-perfect code in 
$\mathbb{Z}_4^{2m}\times \mathbb{Z}_2^{2n'}\times \mathbb{Z}_4^{2n''}$ with $\DD{m}{n'+n''}$-metric.
\end{corollary}
\begin{proof} It remains to note that if $n''$ meets (\ref{eq:3}), 
 then $A'$ has at least $n''$ columns with zeros 
 in the last $\gamma$ positions.
\end{proof}
In general, existence of additive $1$-perfect codes in the case when $m$, $n'$, $n''$ satisfy 
(\ref{eq:1})--(\ref{eq:3}) with odd $\Delta$ remains unsolved.
In the next section, we construct one such code.


%
%
%
\section{An additive code with $\Delta=3$}\label{s:3}
In this section, we construct an additive code in 
$\mathbb{Z}_4^{14}\times \mathbb{Z}_4^{7}$ 
that is $1$-perfect in $\DD{7}{7}$.
The check matrix is
$$\left(
\begin{array}{c@{\,}cc@{\,}cc@{\,}cc@{\,}cc@{\,}cc@{\,}cc@{\,}c|ccccccc}
 1&2& 2&2& 0&3& 3&2& 0&3& 1&3& 1&1&  1&0&0&1&2&3&1 \\
 0&3& 3&0& 2&3& 1&1& 3&3& 3&0& 0&2&  0&1&0&3&3&3&2 \\
 2&2& 0&3& 3&2& 0&3& 1&3& 1&1& 1&2&  0&0&1&2&3&1&1
\end{array}\right),
$$
and it can be directly checked that every nonzero syndrome is
covered by one of the seven pairs of left coordinates or
by one of the seven right coordinates.
Below, we briefly show a cyclic representation of this matrix, 
omitting some details and the algebraic background.
The columns of the matrix are considered as vectors
over $Z_4$ that represent elements of the Galois ring
$\mathrm{GR}(4^3)$.
Let $\xi$ be a primitive seventh root of $1$ in $\mathrm{GR}(4^3)$;
then, every element of $\mathrm{GR}(4^3)$ 
is uniquely represented as $a+2b$, 
$a,b\in\{0,\xi^0,\xi^1,{\ldots},\xi^6\}$.
The first $14$ columns of the matrix are divided into the pairs
$\xi^{i}+2\xi^{i+2}$,  $\xi^{i+1}+2\xi^{i+5}$, $i=0,1,\ldots,6$;
the last $7$ columns are $\xi^0$, $\xi^1$, \ldots, $\xi^6$.
It can be checked that
the syndromes $\xi^{i}+2\xi^{i+2}$,
$\xi^{i+1}+2\xi^{i+5}$, $\xi^{i+3}+2\xi^{i+6}$,
$\xi^{i}+2\xi^{i+6}$, $\xi^{i+1}+2\xi^{i+6}$, $\xi^{i+3}+2\xi^{i+4}$,
are covered by the pair of coordinates $(2i+1,2i+2)$, $i=0,1,2,3,4,5,6$. 
We see that $\xi^j+\xi^{j+k}$ occurs for every $k=1,2,3,4,5,6$.
The syndromes $\xi^{i}$,  $2\xi^{i}$, and $\xi^{i}+2\xi^{i}$
are covered by the coordinate $15+i$, $i=0,1,2,3,4,5,6$.
So, every nonzero syndrome 
$a+2b$, $a,b\in\{0,\xi^0,\xi^1,{\ldots},\xi^6\}$, $(a,b)\ne (0,0)$
is covered.
\section{Nonlinear codes}\label{s:nonl}
In this section, 
we use a variant of the product construction 
from \cite{Phelps:q,Mollard} 
to construct $1$-perfect codes in $\DD{m}{n}$, $2m+n=(4^\mu-1)/3$,
for rather wide spectrum of values of $m$.

Let for every $i$ from $1$ to $k$ and $j$ from $1$ to $r$,
$f_{i,j},g_{i,j}:\mathbb{E}_2^3 \to \mathbb{E}_2$
be two functions such that the set 
\begin{equation}\label{eq:Cij}
C_{i,j}=\{(\bar x,f_{i,j}(\bar x),g_{i,j}(\bar x))
\,:\, \bar x \in \mathbb{E}_2^3\}
\end{equation}
is a $1$-perfect code in the Hamming graph 
$H(5,4)=K\times K\times K\times K\times K$.
Let us define two \emph{generalized parity-check functions} on 
$(\mathbb{E}_2^3)^{kr}$:
\begin{eqnarray}
 f(\bar x_{1,1},{\ldots},\bar x_{k,r}) &=&
(f_1(\bar x_{1,1},{\ldots},\bar x_{1,r}),{\ldots},
 f_k(\bar x_{k,1},{\ldots},\bar x_{k,r})),\quad\mbox{where} \label{eq:f} \\
 f_i(\bar x_{i,1},{\ldots},\bar x_{i,r})&=&
 f_{i,1}(\bar x_{i,1})+{\ldots}+f_{i,r}(\bar x_{i,r});\nonumber\\
 g(\bar x_{1,1},{\ldots},\bar x_{k,r}) &=&
(g_1(\bar x_{1,1},{\ldots},\bar x_{k,1}),{\ldots},
 g_r(\bar x_{1,r},{\ldots},\bar x_{k,r})),\quad\mbox{where}  \label{eq:g} \\
 g_j(\bar x_{1,j},{\ldots},\bar x_{k,j})&=&
 g_{1,j}(\bar x_{1,j})+{\ldots}+ g_{k,j}(\bar x_{k,j}).\nonumber
\end{eqnarray}
\begin{lemma}[\cite{Phelps:q,Mollard}]\label{l:mol}
Let $C'$ and $C''$ be two
$1$-perfect codes in 
$\mathbb{E}_2^k$ and $\mathbb{E}_2^r$, 
respectively, with the Hamming metric. 
And let $f$, $g$ be the generalized parity check functions 
defined as above. Then the set 
\begin{equation}
 C = \{ (\bar{\bar x},f(\bar{\bar x})+c', g(\bar{\bar x})+c''\,:\,
\bar{\bar x} \in (\mathbb{E}_2^3)^{kr}, c'\in C', c''\in C'' \}
\end{equation}
is a $1$-perfect code in the Hamming graph 
$H(3kr+k+r,4)=(K^3)^{kr}\times K^k \times K^r$.
\end{lemma}
Now, let us change the definition of the first $m$ pairs 
$(f_{i,j},g_{i,j})$ requiring the code
\begin{equation}\label{eq:Cij-mod}
C_{i,j}=\{(\bar x,f_{i,j}(\bar x),g_{i,j}(\bar x))
\,:\, \bar x \in \mathbb{E}_4\times\mathbb{E}_2\}
\end{equation}
to be $1$-perfect
in $\DD{1}{3}$. 
The functions $f$ and $g$ on $(\mathbb{E}_4\times\mathbb{E}_2)^m\times(\mathbb{E}_2^3)^{kr-m}$ defined by the same formulas 
(\ref{eq:f}), (\ref{eq:g}) but with new $f_{i,j}$, $g_{i,j}$ 
will be called \emph{modified generalized parity check functions}.
\begin{lemma}\label{l:mol2}
Let $C'$ and $C''$ be two
$1$-perfect codes in 
$\mathbb{E}_2^k$ and $\mathbb{E}_2^r$, 
respectively, with the Hamming metric. 
And let $f$, $g$ be the modified generalized parity-check functions
defined as above. 
Then the set 
\begin{equation}
 C = \{ (\bar{\bar x},f(\bar{\bar x})+c', g(\bar{\bar x})+c'')\,:\,
\bar{\bar x} \in (\mathbb{E}_4\times\mathbb{E}_2)^m\times(\mathbb{E}_2^3)^{kr-m}, c'\in C', c''\in C'' \}
\end{equation}
is a $1$-perfect code in  the graph
$(\Sh\times K)^m\times (K^3)^{kr-m}\times K^k \times K^r$.
\end{lemma}
\begin{proof}
It is easy to count that the cardinality of $C$ equals
the cardinality of the space divided by the cardinality 
$(3k+1)(3r+1)$ of a ball of radius $1$. So, it is sufficient to show
that every vertex is within radius $1$ from some code vertex.
Since $C'$ and $C''$ are $1$-perfect codes,
every vertex $X$ is representable in the form
$$(\bar{\bar x},f(\bar{\bar x})+c'+e', g(\bar{\bar x})+c''+e'')$$
where $c'$, $c''$ are codewords of $C'$, $C''$, respectively, 
$e'$, $e''$ are of weight at most $1$.
If $e'=\overline 0$ or $e''=\overline 0$, 
then $X$ is at distance $0$ or $1$ from the codeword 
$(\bar{\bar x},f(\bar{\bar x})+c', g(\bar{\bar x})+c'')$.
It remains to consider the case $e',e''\not=\overline 0$.
Let $e'$, $e''$ have nonzero values $y'$, $y''$ 
in the $i$th and $j$th positions, respectively.
Consider the tuple 
$\bar y=(\bar x_{i,j},f_{i,j}(\bar x_{i,j})+y',g_{i,j}(\bar x_{i,j})+y'')$,
where $\bar x_{i,j}$ is the $ij$th block of the tuple $X$.
Since the code $C_{i,j}$ is $1$-perfect, there is 
$\bar z=(\bar v,f_{i,j}(\bar v),g_{i,j}(\bar v))$
such that $\bar y$ is at distance $1$ from $\bar z$
(note that these arguments are independent of the metric space
$\bar x_{i,j}$ and $\bar v$ belong to; 
it can be $\DD{0}{3}$, $\DD{1}{1}$, or even any other metric space provided
(\ref{eq:Cij}),(\ref{eq:Cij-mod}) is a $1$-perfect code).
Clearly, $\bar y$ and $\bar z$ differ in the parts
$\bar x_{i,j}$, $\bar v$ and coincide in the last two positions.
Then, replacing $\bar x_{i,j}$ 
by $\bar v$ in $X$ results in a code vertex from $C$ 
at distance $1$ from $X$.
\end{proof}

It remains to note the fillowing:
\begin{lemma}\label{l:fg}
There are functions $f^0$, $g^0$: $\mathbb{E}_2^3 \to \mathbb{E}_2$ and $f^1$, $g^1$: $\mathbb{E}_4\times\mathbb{E}_2 \to \mathbb{E}_2$ such that the sets 
$
\{(\bar x,f^0(\bar x),g^0(\bar x))
\,:\, \bar x \in \mathbb{E}_4^3\}
$ 
and
$
\{(\bar x,f^1(\bar x),g^1(\bar x))
\,:\, \bar x \in \mathbb{E}_4\times\mathbb{E}_2\}
$ 
are $1$-perfect codes in $\DD{0}{5}$ and $\DD{1}{3}$, respectively.
\end{lemma}
\begin{proof}
The existence of functions $f^i$, $g^i$ follows directly 
from the existence of $1$-perfect codes in the corresponding graph 
(\cite{Golay:49} and \cite{KoolMun:2000}, respectively). To be explicit,
we suggest direct formulas:
$$
\begin{array}{lll}
f^0(x,y,z) = x+y+z, &\quad 
g^0(x,y,z) = x+\omega y+\oomega z, 
&\quad x,y,z\in \mathbb{E}_2; \\[0.7ex]
 f^1(\omega x + y, \varphi(z)) = \varphi(x+y+z), &\quad
 g^1(\omega x + y, \varphi(z)) = \varphi(x+2y+3z),
 &\quad x,y,z\in \mathbb{Z}_4,
\end{array}
$$
where $\varphi$ is any bijection between the elements of $\mathbb{Z}_4$ and $\mathbb{E}_2$.
\end{proof}
Finally, we can state the following.
\begin{theorem}\label{th:nonlin}
Assume that positive integers $m$, $n$, $\mu$
satisfy 
\begin{eqnarray}
2m+n &=& ({4^\mu-1})/3, \nonumber\\ \label{eq:m-kr}
m&\le&
\left\{\begin{array}{ll}
(4^\mu - 2.5\cdot 2^\mu+1)/9 &\ \mbox{if $\mu$ is odd}, \\[0.7ex]
(4^\mu - 2\cdot 2^\mu+1)/9 &\ \mbox{if $\mu$ is even}.
\end{array}\right.
\end{eqnarray}
Then there is a $1$-perfect code in the Doob graph $\DD{m}{n}$.
\end{theorem}
\begin{proof}
By Lemma~\ref{l:mol2}, we can construct 
a $1$-perfect code in a graph
$(\Sh\times K)^m\times (K^3)^{kr-m}\times K^k \times K^r$
isomorphic to $\DD{m}{n}$, where 
$k=(2^{\mu-1}-1)/3$, $r=(2^{\mu+1}-1)/3$
or $k=r=(2^{\mu}-1)/3$, depending on the parity of $\mu$.
The condition $m\le kr$ is guaranteed by (\ref{eq:m-kr}).
\end{proof}
\section{Open problems}\label{s:fin}
\begin{problem} For every value $(m,n)$ satisfying $2m+n = (4^\mu-1)/3$ and 
not covered by the constructions in
Sections~\ref{s:constr}--\ref{s:nonl}, construct a $1$-perfect code in $\DD{m}{n}$
or prove its nonexistence.
In particular, do there exist $1$-perfect codes in $\DD{6}{9}$, $\DD{9}{3}$, $\DD{10}{1}$?
\end{problem}
\begin{problem}
For every value $(m,n',n'')$ satisfying (\ref{eq:1})--(\ref{eq:3}) with odd $\Delta\ge 3$
(except the case $(7,0,7)$, considered in Section~\ref{s:3}),
construct an additive $1$-perfect code in
$\mathbb{Z}_4^{2m}\times\mathbb{Z}_2^{2n'}\times\mathbb{Z}_4^{n''}$ with 
the $\DD{m}{n'+n''}$-metric or prove its nonexistence.
In particular, does there exist an additive $1$-perfect code 
in $\mathbb{Z}_4^{16}\times\mathbb{Z}_2^{2}\times\mathbb{Z}_4^{4}$ 
with the $\DD{8}{5}$-metric?
\end{problem}

\providecommand\href[2]{#2} \providecommand\url[1]{\href{#1}{#1}}
 \def\DOI#1{{\small {DOI}:
  \href{http://dx.doi.org/#1}{#1}}}\def\DOIURL#1#2{{\small{DOI}:
  \href{http://dx.doi.org/#2}{#1}}}

\end{document}